\documentclass[12pt]{amsart}

\usepackage{amsmath}
\usepackage{mathtools}
\usepackage{latexsym}
\usepackage{subfigure}
\usepackage{appendix}
\usepackage{amsthm}
\usepackage[normalem]{ulem}
\usepackage{graphicx}
\usepackage{psfrag}
\usepackage[usenames,dvipsnames]{color}
\usepackage{enumerate}
\usepackage{url}
\usepackage{mathtools}
\usepackage{float}

\newcommand{\excise}[1]{}%{$\star$\textsc{#1}$\star$}

\newtheorem{thm}{Theorem}[section]
\newtheorem{lemma}[thm]{Lemma}

\newtheorem{cor}[thm]{Corollary}
\newtheorem{prop}[thm]{Proposition}

\theoremstyle{definition}

\newtheorem{example}[thm]{Example}
\newtheorem{remark}[thm]{Remark}
\newtheorem{defn}[thm]{Definition}

\newtheorem{case*}[thm]{Case}

\numberwithin{equation}{section}

\begin{document}%%%%%%%%%%%%%%%%%%%%%%%%%%%%%%%%%%%%%%%%%%%%%%%%%%%%%%%%
%%%%%%%%%%%%%%%%%%%%%%%%%%%%%%%%%%%%%%%%%%%%%%%%%%%%%%%%%%%%%%%%%%%%%%%%

\mbox{}
%\vspace{-2ex}%-1.1743pt}
\title{Binomial ideals of domino tilings }

\author{Elizabeth Gross and Nicole Yamzon }
\thanks{$\ $ \\
Elizabeth Gross was supported by the National Science Foundation DMS- 1620109 and DMS-1945584. Nicole Yamzon was supported by Alfred P. Sloan
Foundation’s Minority Ph.D. Program, awarded in (2018)}
\address{Department of Mathematics, University of Hawai`i at M\={a}noa, Honolulu, HI 96822}
\email{egross@hawaii.edu}
\address{Department of Mathematics\\University of Illinois\\Urbana, IL 61801}
\email{nyamzon2@illinois.edu}

\date{\today}

\begin{abstract} In this paper, we consider the set of all domino tilings of a cubiculated region. The primary question we explore is: \emph{How can we move from one tiling to another?} Tiling spaces can be viewed as spaces of subgraphs of a fixed graph with a fixed degree sequence.  Moves to connect such spaces have been explored in algebraic statsitics. Thus, we approach this question from an applied algebra viewpoint, making new connections between  domino  tilings,  algebraic  statistics, and toric algebra.  Using results from toric ideals of graphs, we are able to describe moves that connect the tiling space of a given cubiculated region of any dimension. This is done by studying binomials that arise from two distinct domino tilings of the same region. Additionally, we introduce \emph{tiling ideals} and \emph{flip ideals} and use these ideals to restate what it means for a tiling space to be flip connected.   Finally, we show that if  $R$ is a $2$-dimensional simply connected cubiculated region,  any binomial arising from two distinct tilings of $R$ can be written in terms of quadratic binomials.  As a corollary to our main result,  we obtain an alternative proof to the fact that the set of domino tilings of a $2$-dimensional simply connected region is connected by flips.

\end{abstract}

\maketitle

% \setcounter{tocdepth}{1}
% \tableofcontents

%%%%%%%%%%%%%%%%%%%%%%%%%%%%%%%%%%%%%%%%%%%%%%%%%%%%%%%%%%%%%%%%%%%%%%%%%

\section{Introduction and Background}
A $2 \times 1$ \emph{domino} (or a $1 \times 2$ domino) is two unit squares joined along a single edge. A \emph{domino tiling} of a region is a covering of the region with dominos such that there are no gaps or overlaps. As an area of mathematical research, domino tilings  appeared as early as  1937 in the context of thermodynamics and dimer systems \cite{fowl}. Several bodies of work in the 2-dimensional setting show that such objects are rich and nuanced \cite{k} \cite{ft}, \cite{t}, \cite{ebb}, \cite{ken}, \cite{ken1}, \cite{ken2}. For example,  Kasteleyn and Fisher--Temperley proved independently \cite{ft, k} that the number of tilings of a $2n \times 2m$ rectangle is 
$$
4^{mn} \prod_{j=1}^m \prod_{k = 1}^n \Big( \cos^2 \frac{2\pi}{2m+1} +\cos^2 \frac{k \pi}{2n+1} \Big).
$$
While, there is no closed-form expression for the number of domino tilings of an arbitrary $2$-dimensional region, there are many papers that study this problem for specific types of regions, such as Aztec diamonds and pyramids \cite{sylv, stanp, stan}.  Higher dimensional regions, such as $3$-dimensional regions, have also been explored.  For example, in 1998, Ciucu gave an upper bound on the number of $3$-dimensional domino tilings of a $n \times n \times n$ cube \cite{ciu}.

In lieu of a complete enumeration of the tilings of a region, one can estimate the number of tilings through Monte Carlo Markov chain sampling methods. Such methods require a set of moves that connect the space, which lead us to our primary question of interest: \emph{What are sets of moves that connect the space of domino tilings for a fixed region?} We tackle this question from an algebra lens, making a new connection between domino tilings,  algebraic statistics, and combinatorial commutative algebra through toric ideals of graphs.

For $2$-dimensions, it is known that any two domino tilings of a simply connected region can be obtained through a sequence of flips \cite{t, sal}.  Generalizing results to higher dimensions has been of interest to fields from combinatorics to solid state chemistry \cite{ebb, ken2}. However, previous results, such as those built on Thurston's \emph{height function} \cite{ran} fail to generalize to the $3$-dimensional setting. In fact, in $3$-dimensions even the most uncomplicated of regions fail to be flip connected. For instance the domino tilings of $l \times m \times n$ box are not connected by flips \cite{fkms,k,milet}. Milet and Saldanha  introduced another local move called the \emph{trit}, which operates on three dominoes at a time (as opposed to the flip that operates on two dominoes at a time) \cite{milet}.  While some $3$-dimensional tiling spaces are connected by flips and trits, not all are; \cite{milet} includes some examples. In  \cite{fkms}, Klivans et al. give conditions in terms of topological invariants for testing whether two different $3$-dimensional domino tilings are connected by flips or by flips and trits.  %{\color{magenta} maybe another sentence on what is know in the 3-D case}

Our paper explores the connectivity question by noting that the space of tilings of a region $R$ corresponds to a particular \emph{fiber} of a \emph{design matrix} $A$, where $A$ is prescribed by the region $R$ (here we are using language from algebraic statistics, which is defined formally in the Section \ref{sec:toricideals}).  By appealing to algebraic statistics, and in particular, the Fundamental Theorem of Algebraic Statistics \cite{algstatsbook}, for any given region, \emph{of any dimension}, we can find a set of moves that is guaranteed to connect its tiling space by finding a set of generators of the \emph{toric ideal of A}, a binomial ideal.  Using the well-known correspondence between domino tilings of a region $R$ and perfect matchings of an associated graph $G_R$ \cite{west} and results in combinatorial commutative algebra on \emph{toric ideals of graphs} \cite{oh}, in Theorem \ref{thm:generators}, we describe the moves guaranteed to connect the tiling space in terms of the graph $G_R$.  The moves described in  Theorem  \ref{thm:generators} are not always local flips though.  While we can show that if the toric ideal of $G_R$ is quadratic, then the space of tilings of $R$ is flip connected, the converse is not always true. In order to explore flip connected tiling spaces more, we introduce two additional binomial ideals, the \emph{tiling ideal} and the \emph{flip ideal}.  These ideals are not always prime, and thus not always toric, however we can describe exactly when a tiling space is flip connected using these two ideals (Theorem \ref{thm:flipconnected}). Finally, we showcase this algebraic perspective by providing an alternative proof to the fact that the tiling space of any simply connected region of $\mathbb{R}^2$ is flip connected. 

This paper is organized as follows. In Section 2, we formally define domino tilings and discuss tilings from a graph theoretic viewpoint. In Section 3, we discuss the connection of tiling spaces  to algebraic statistics and toric ideals of graphs and describe a set of moves guaranteed to connect the tiling space of a given region.  Additionally, we introduce the tiling ideal and the flip ideal of a region $R$ and restate what it means for a tiling space to be connected in terms of these two ideals.  Finally, in Section 4, we show, using algebraic techniques, that the tiling space of any simply connected region of $\mathbb{R}^2$ is flip connected.

\section{Domino Tilings}
Let $R$ be a $n$-dimensional cubiculated region of $\mathbb R^{N}$, i.e. a homogeneous cubical complex of dimension $n$ embedded in $\mathbb R^{N}$ with $n \leq N$. A $d$-dimensional \emph{domino} is two adjacent elementary $n$-dimensional cubes connected along a face of dimension $n-1$; we will denote the set of dominos contained in $R$ as $\mathcal D_R$.
A \emph{domino tiling} $T \subseteq \mathcal D_R$ of a cubiculated region $R$ is defined to be a covering of $R$ with dominoes such that every elementary cube of $R$ is covered  exactly once; we will denote the space of all tilings of $R$ as $\mathcal T_{R}$.  We are interested in local moves that connect all the tilings in $\mathcal T_{R}$. A move between two tilings $T_1, T_2 \in \mathcal T_{R}$ is an ordered pair $M=(D_1, D_2)$ of two sets of dominoes $D_1, D_2 \subseteq \mathcal D_R$ such that $T_2 = (T_1\setminus D_1) \cup D_2$.  For simplicity, if $M$ is a move from $T_1$ to $T_2$, we will write $T_2 = T_1 +M$. %{\color{magenta} In general, a move for $R$ is a pair $(D_1, D_2)$ such that $D_1, D_2 \subseteq \mathcal D_R$ and $D_1$ and $D_2$ are both tilings of the same subregion of $R$.}  
We say a move has size $d$ if $|D_1|=|D_2|=d$. We begin our discussion with the simplest move, the \emph{local flip}, or \emph{flip}.

%We are interested in a set of moves that connects that  The main mathematical tool used to analyze the underlying structure of connectivity will be the theory of toric ideals. We will focus on the world of  $2$-dimensional domino tilings. In terms of connectivity, various methods have been used to characterize the types of ``local moves" from one tiled region to another. One of these moves is called a \emph{local flip}, or equivalently, a \emph{flip}.
%[TODO: EXPAND INTO A 2 PARAGRAPH LITERATURE REVIEW]Given a region, several natural tiling questions arise. First, does a tiling exist? If so, how many are there? And lastly, how do we move from one to another? 

%[In terms of existence, it has been shown $\ldots$. Counting tilings is also interesting.  Here are some results.] 

\begin{defn}
A \emph{local flip} is performed by replacing a pair of two adjacent parallel dominoes, i.e. two dominoes that share two $n-1$ dimensional faces, with two adjacent parallel dominoes in a perpendicular direction to the first pair.  We will refer to the set of all flip moves for $R$ as $\mathcal M_{R_{flip}}$.
\begin{figure}[h]
    \centering
    \includegraphics[scale=.25]{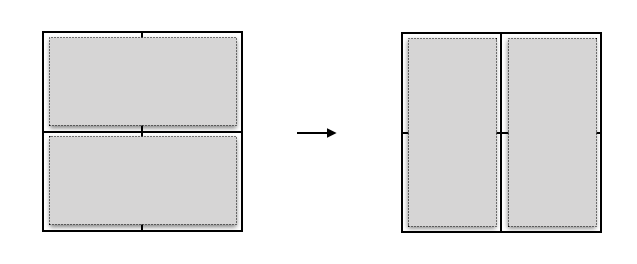}
    \caption{A local flip.}
    \label{fig:flip}
\end{figure}
\end{defn}

\begin{defn}
Let $R$ be a cubiculated region, and let $\mathcal M_R$ be a set of possible moves of $R$. We say $\mathcal M_R$ \emph{connects} $\mathcal T_R$ if for every two tilings $T_1, T_2 \in \mathcal T_{R}$, there exists $M_1, \ldots M_r \in \mathcal M_R$ such that $T_2 = T_1 + M_1+M_2 + \ldots M_r$ and $T_1 + M_1+\ldots M_s \in \mathcal T_R$ for all $1 \leq s \leq r$. 
\end{defn}

Two tilings $T_1$ and $T_2$ are \emph{flip connected} if there exists a sequence $M_1, \ldots, M_r \in \mathcal M_{R_{flip}}$ such that $T_2 = T_1 + M_1 + \ldots M_r$ and $T_1 + M_1+\ldots M_s \in \mathcal T_R$ for all $1 \leq s \leq r$. A cubiculated region $R$ is \emph{flip connected} if every two tilings of $R$ are flip connected. It is known that if a $2$-dimensional region is simply connected then its corresponding space of tilings is flip connected.

\begin{thm} \cite{t, sal} If $R$ is a simply connected region in $2$-dimensions, then $\mathcal T_R$ is flip connected.\label{thm:Thurston}
\end{thm}
In \cite{t}, and more explicitly in \cite{sal}, Theorem \ref{thm:Thurston} is proved via the construction and analysis of a map from the vertices of a domino tiling to $\mathbb{Z}$ called the \emph{height function}.  In this paper, we give an alternative proof of the theorem using binomial ideals.
 
%In this thesis, we give an alternative proof of Theorem \ref{thm:Thurston} using algebraic methods.

%[In addition to Thurston's result, 2-d tilings have been further studied in $\ldots$.]\\

%[Three-dimesional tilings have been a bit trickier.  For example, 3-d boxes are not flip connected (citation).  However, Klivans et al. have shown that $\ldots$ are flip and trit connected.]

%\begin{defn}
%A \emph{trit} is the local move performed by removing and replacing three dominos. Where each domino is parallel to one axis inside the $(2 \times 2 \times 2)$-box.
%\end{defn}

%[ADD PICTURE ILLUSTRATING TRIT MOVE]\\
%\\
%A region is \emph{flip and trit connected} if we can move between every two tilings with a sequence of flip and trit moves.  It is conjectured that 3-d boxes are flip and trit connected.

%\begin{conj}[CITATION]  Let $R$ be a $m \times n \times \ell$ box.  Then $R$ is flip and trit connected.
%\end{conj}

%[TODO: Say something about the contributions of the thesis here.]

\subsection{Connections to graph theory}

In order to apply tools from combinatorial commutative algebra, it is helpful to think of tilings of a cubiculated region as perfect matchings of a graph.  Here we set up the terminology to construct this correspondence.

Let $R$ be a $n$-dimensional cubiculated region.  Let $G_R$ be the undirected simple graph that has one vertex for each elementary cube in $R$ and an edge between a pair of vertices if their two corresponding cubes in $R$ share a $n-1$ dimensional face. Given a graph $G = (V, E)$, a \emph{matching} $M$ is an independent edge set. A \emph{perfect matching} is a matching that covers all vertices in $G$. By the construction of $G_R$ from $R$, we see that there is a one-to-one correspondence between tilings of $R$ and perfect matchings on the graph $G_R$. This correspondence is illustrated in Figure \ref{fig:perfmatch}.

%\begin{center}
 %   %\caption{A $(4\times 5)$-box.}\\
  %  \includegraphics[scale =.5]{4x5grid.png}
%\end{center}

\begin{figure}[h]

    \includegraphics[scale=.25]{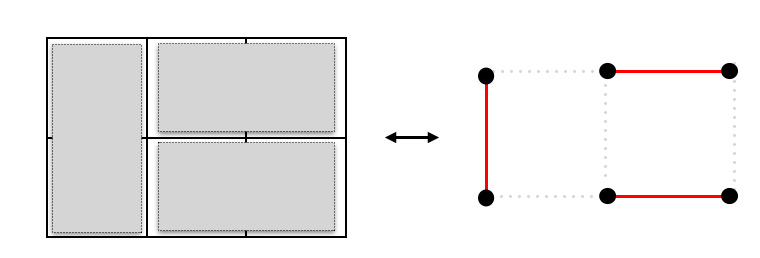}
    \caption{Let $R=B_{2,3}$, the $2 \times 3$ box. On the left is a tiling on $R$, on the right is a perfect matching of $G_{2,3}$.}
    \label{fig:perfmatch}
\end{figure}

\begin{example} Let the region $R$ be the $m \times \ell$ box denoted $B_{m,\ell}$. Then $G_R$ is denoted by $G_{m,\ell}$ and is the $m \times \ell$ grid graph whose vertices correspond to the points in $[0,m] \times [0,\ell] \cap \mathbb{Z}^2$. 
\end{example}

\begin{remark} \label{rmk:gridgraphs}
Since $R$ can always be viewed as a subregion of a $n$-dimensional box $B_{m_1, \ldots, m_n}$, the graph $G_R$ is a subgraph of the grid graph $G_{m_1, \ldots, m_n}$. 
\end{remark}

For the rest of this paper, we will refer to tilings and matchings interchangeably. We will use the underlying graph structure to understand the connectivity of the space of $2$-dimensional domino tilings for a region $R$. While tilings on a cubiculated region $R$ can be characterized by perfect matchings on $G_R$, moves between two tilings can be characterized by even cycles. 

%\begin{defn}
%An \emph{$n$-cycle} of a graph $G$ is a subset of edges that form a sequence of length $n$ 

%$$
%(e_1, e_2, \dots, e_n)
%$$

%such that $(e_1) = (e_n)$. A cycle $w$ can be of two types: 
%\begin{enumerate}
  %  \item[(1)] Where $w$ is a \emph{closed walk} where all edges in the cycle may be traversed multiple times in $w$ and $(e_1) = (e_n)$.
    
   % \item[(2)] Or $w$ is not a closed walk and the only edges that repeat are the first and last in $w$. A walk that is not closed can also be called a \emph{simple cycle}.
%\end{enumerate}
%\end{defn}

\begin{defn}
A \emph{walk} on $G$ is a finite sequence of the form 
$$
w = \big((v_{1}, v_{2}), (v_{2}, v_{3}), \dots, (v_{n-1}, v_{n})\big)
$$
with each $v_{i} \in V(G)$ and $\{v_{i-1}, v_{i}\} \in E(G)$. In the case where $v_{1} = v_{n}$, then $w$ is called a \emph{closed walk}.  A \emph{cycle} is a closed walk that traverses each vertex in the walk exactly once.  The \emph{length} of a cycle or closed walk is the number of edges in the walk. A closed walk is \emph{even} if the cycle has even length. An even closed walk is \emph{primitive} if it does not contain a proper closed even subwalk. 
\end{defn}

%\begin{remark}
%All edges are undirected in this paper. 
%\end{remark}

%\begin{remark} We will use the notation that  
%$$
%\big((v_{1}, v_{2}), (v_{2}, v_{3}), \dots, (v_{{n-1}}, v_{n})\big) = \big( e_{1}, e_{2}, \dots, e_{n-1} \big)
%$$
%interchangeably. 
%\end{remark}

\begin{prop}
Let $R$ be a cubiculated region. Every cycle of the graph $G_R$ is even.
\end{prop}

\begin{proof} As implied by Remark \ref{rmk:gridgraphs}, for any cubiculated region $R$, the graph $G_R$ is a subgraph of a grid graph $G_{m_1, \ldots, m_n}$.  In particular, $G_R$ is bipartite.  Thus, every cycle of $G_R$ is even \cite{west}.
\end{proof}
%{\color{magenta} double-check for higher dimensions}

\begin{remark} Since every cycle of $G_R$ is even, the only primitive even closed walks on $G_R$ are cycles \cite{oh}.
\end{remark}

In the following proposition, we see that the union of two tilings of $R$ corresponds to a collection of cycles on $G_R$.

\begin{prop}
Let $T_1$ and $T_2$ be tilings and let $G = T_1 \cup T_2$ (considered as a multigraph). Then $G$ will be a disjoint collection of even cycles, some of which may be $2$-cycles. 
\end{prop}

\begin{proof}
Consider the graph $G= T_1 \cup T_2$ with $n$ vertices. We know for each $i \in V(G)$ the $deg_{G} (i) =2$. By definition $G$ must be a $2$-regular graph of size $n$. A characterization of $2$-regular graphs gives us that $G$ will be formed by a disjoint collection of cycles.
\end{proof}

Since $G=T_1 \cup T_2$ is a disjoint collection of cycles we introduce the following terminology.

\begin{defn}
Let $G$ be a graph. We say $\mathcal C = \{\mathcal C_1, \mathcal C_2, \ldots, \mathcal C_r\}$ is a \emph{cycle cover} of $G$ if each $\mathcal C_i$ is a cycle and every vertex in $G$ is covered by exactly one $\mathcal C_i$.
\end{defn}

Note that given a region $R$ and two tilings, $T_1$ and $T_2$, the multigraph $\mathcal C = T_1 \cup T_2$ is a cycle cover of $G_R$.  Additionally, we can think of the edges in $\mathcal C$ as two-colorable, specifically, we can color the edges corresponding to $T_1$ red and the edges corresponding to $T_2$ as blue.  This coloring will be helpful in later sections.

Finally, we end this section with a discussion on chords, which will play a role in the algebra in the next two sections.

\begin{defn} Let $\mathcal{C}$ be a cycle of a graph $G=(V,E)$. An edge $e \in E$ is a \emph{chord} of $\mathcal{C}$ if $e$ connects two vertices covered by $\mathcal{C}$, but is not in $\mathcal{C}$. A cycle is \emph{chordless} if it does not have a chord in $G$.
\end{defn}

\begin{defn} Let $\mathcal{C}=\big((v_{1}, v_{2}), (v_{2}, v_{3}), \dots, (v_{n-1}, v_{n})\big)$ be an even cycle of a graph $G=(V,E)$. We say $e=\{v_i, v_j\} \in E(G)$ is a chord of $\mathcal{C}$ with $i<j$. Furthermore, we call $e$  an \emph{even chord} if $j-i$ is odd, in other words, if the two new cycles obtained by adding $e$ to $\mathcal{C}$ are both even.
\end{defn}

%\subsection{An application in two-dimensions}
%\hfill \break 
%{\color{blue} maybe add a shortened version of this to the introduction.} Although we understand this problem entirely in the context of combinatorics, being able to fully characterize the moves of two-dimensional domino tilings has applications to physics. For example, there is a direct connection to the {\bf dimer model}, a model that captures natural phase transitions in material science.

%\begin{defn}
%A \emph{dimer} is a polymer with only two atoms. 
%\end{defn}

%\begin{defn}
%A \emph{dimer covering} is a \emph{perfect matching} of a graph $G$, where the vertices of $G$ represent univalent atoms and the edges represent bonds between atoms.  The \emph{dimer model} studies collections of dimer coverings.
%\end{defn}

%In the language of graph theory we can also call dimer coverings in more general terms called perfect matchings. 

%Given any medium, the phase transitions as a physical phenomena can become quite complicated to model due to unaccounted external conditions such as temperature, pressure, etc. Even though the dimer model is simple, it is powerful because of our current understanding of the space of two-dimesional tilings. %the two-dimensional flip connected space provides a model that can be studied on an analytic level. 
%\noindent In three-dimensions this no longer remains the case.
\medskip

%\begin{thm}There exists a one-to-one correspondence with dominio tilings of an $(m \times n)$-box and perfect matchings in a $(m \times n)$-grid graph. 
%\end{thm}

\section{Tilings and Toric Ideals of Graphs} \label{sec:toricideals}

In this section, we introduce toric ideals of graphs and their connections to tiling spaces. Toric ideals of graphs have been well-studied (see, for example, \cite{v, oh, oh2, reyes, tatakis,biermann, galetto}).  By making the connection to toric ideals of graphs, we can describe a set of moves that is guaranteed to connect the tiling space $\mathcal{T}_R$ for any cubiculated region $R$.

\subsection{Toric ideals of graphs and Markov bases}

Let $G=(V, E)$ be a graph. Consider the following two polynomial rings
\begin{align*}
\mathbb{K}[E] &= \mathbb{K} [ y_e \mid e \in E[G]], \text{ and}\\
\mathbb{K}[V]  & = \mathbb{K} [x_v \mid v \in V(G) ].
\end{align*}  
Let $\phi_G$ be the ring homomorphism defined as follows
\begin{align*}
 \phi_{G}: \mathbb{K}[E] &\to \mathbb{K}[V] \\ y_{(i,j)} & \mapsto x_ix_j.  
\end{align*} 
The \emph{toric ideal} of $G$, denoted $I_G$, is defined to be the kernel of the map $\phi_G$ 
$$I_G : = \ker (\phi_G) = \{ f \in \mathbb{K}[E] \ : \ \phi_{G}(f) = 0 \}.$$

  For our application, we are going to be most interested in the generating set of a toric ideal of a graph.  Such generating sets can be described by primitive closed even walks.  Furthermore, when $G$ is bipartite, a generating set of $I_G$ can be given simply in terms of even cycles.

\begin{defn}
Let $w$ be an even cycle, i.e.
$$
w = \Big((v_1,v_2), (v_2, v_3), \dots, (v_{2n}, v_{1})\Big)
$$
where $e_i = \{v_i,v_{i+1}\}$. The binomial arising from $w$ is
$$
B_w = \prod_{i=1}^n y_{e_{2i-1}}-\prod_{i=1}^ny_{e_{2i}}.
$$

\end{defn}

\begin{prop} Given a bipartite graph $G$, the ideal $I_G$ is generated by the set of binomials arising from even cycles on $G$ \cite{v}.
\end{prop}

%\begin{remark}Walks constructed from two $k$-domino tilings $T_1$ and $T_2$ of the $m \times n$-box will be of the form
%$$
%w = (e_1, e_2,  \dots, e_{2k-1}, e_{2k})
%$$
%where $e_{2i-1} \in T_1$ and $e_{2i} \in T_2$ where $ 1\leq i \leq k$.
%\end{remark}

%We can express $w$ with the following binomial

Just as we can define a binomial arising from a cycle, we can define a binomial associated to two tilings. Regard two tilings $T_1$ and $T_2$ of a cubiculated region $R$ as perfect matchings in $G_R$.

\begin{defn}
Define the binomial arising from $(T_1,T_2)$ to be 
$$
B_{T_1,T_2} = \prod_{e_i \in T_1} y_{e_i} - \prod_{e_j \in T_2} y_{e_j}.
$$
%which is equivalent to an edge labeling 
%$$
%= \prod_{i=1}^k {e_{2i-1}} - \prod_{i=1}^k e_{2i}.
%$$
\end{defn}
\noindent Note that if $R$ is a cubiculated region and $T_1, T_2 \in \mathcal T_{R}$, then the binomial $B_{T_1, T_2}$ arising from $(T_1, T_2)$ is in the toric ideal $I_{G_R}$ since
$$
\phi_{G_R}(B_{T_1,T_2})= \prod_{e_i \in T_1}\phi_{G_R}\big( y_{e_i}\big)-\prod_{e_j \in T_2}\phi_{G_R}\big(  y_{e_j} \big)=\prod_{i \in V(G_{R})} x_{i} - \prod_{i \in V(G_{R})} x_{i} = 0.
$$

\begin{remark}
For ease of notation, we will use the following monomial shorthand.  Let $E_0 \subseteq E(G)$, then we define
$$y^{E_0}:= \prod_{e_i \in E_0} y_{e_i}.$$
Thus, we will write $B_{T_1, T_2}$ as
$$B_{T_1, T_2}=y^{T_1}-y^{T_2}.$$
\end{remark}

%\begin{remark}
%Observe that two tilings of size $k$ will construct a walk of length $2k$.
%\end{remark}
%\begin{defn}
%A \emph{simple path} is a walk $w = \big(e_{i_1}, e_{i_2}, \dots, e_{i_n}\big)$ where each edge is distinct. 
%\end{defn}

%We can algebraically relate back to toric theory nicely by being able to describe the generators of a generalized toric ideal as binomials of closed even walks.

%\begin{thm}Ohsugi-Higi
%The ideal $I_G$ is generated by binomials corresponding to closed even walks.
%\end{thm}

Similar to a binomial $B_{T_1, T_2}$ arising from two tilings, the \emph{binomial arising from a tiling move} $(D_1, D_2)$ is $B_{D_1, D_2} = y^{D_1} - y^{D_2}$ and is also in $I_{G_R}$.

We can describe a way to move between any two tilings in $\mathcal T_R$ by invoking the Fundamental Theorem of Markov Bases from algebraic statistics \cite{diaconis, algstatsbook}. To do this we now build a connection between toric ideals of graphs and the language of Markov bases.  First, let's describe $I_G$ in an alternate way using design matrices.  Indeed, the most common way to define a toric ideal is through an integer matrix $A$; this matrix is referred to as the \emph{design matrix} in algebraic statistics.  Let $A$ be the vertex-edge incidence matrix of $G$ with $N=\#E(G)$ columns.  Then
$$I_{G} = I_{A} := \langle y^u - y^v \ | \ u, v \in \mathbb Z_{\geq 0}^N, Au = Av \rangle.$$
In this setting, we can think about $u, v \in \mathbb Z_{\geq 0}^N$ as integer vectors or as multisets of edges drawn from $E(G)$.  The condition $Au=Av$ means that $u$ and $v$ have the same degree sequence as multigraphs.

Let $u \in \mathbb Z_{\geq 0}^N$.  The \emph{fiber of $u$ with respect to $A$} is
$$\mathcal F (u)=\{v \in Z_{\geq 0}^N :Av=Au \}.$$
The fiber of $u$ is precisely the collection of all multigraphs with edges drawn from $E(G)$ with same degree sequence as $u$. Since every tiling of $R$ has the same degree sequence when viewed as a perfect matching of $G_R$, it is the case that
$\mathcal T_R = \mathcal F(T)$ for any tiling $T$ of $R$.  This key observation allows us to use Markov bases to find a set of moves to connect $\mathcal T_R$.

\begin{defn}
Let $A \in \mathbb Z^{M \times N}$. Let $\text{ker}_{\mathbb Z} A = \{v \in \mathbb{Z}^N : Av = 0\}$ be the integer kernel of $A$. A finite set $\mathcal{MB} \subset \text{ker}_{\mathbb Z} A$ is called a \emph{Markov basis} for $A$ if for all $u \in \mathbb Z_{\geq 0}^N$ and $v \in \mathcal F(u)$, there is a sequence $b_1,\ldots,b_r \in \pm \mathcal{MB} := \{ (-1)^i b \ : \ b \in \mathcal{MB}, \ i = 0,1\} $ such that
$$v = u + \sum_{k=1}^r b_k \text{ and } u + \sum_{k=1}^s b_k \geq 0 \text{ for all } s=1, \ldots r.$$
The elements of a Markov basis are called \emph{Markov moves}.
\end{defn}

When $A$ is the vertex-edge incidence matrix of $G_R$ for a cubiculated region $R$, every Markov move $b$ with $0, 1$, and $-1$ entries corresponds to a move $M=(D_1, D_2)$ on $R$ by letting $D_1$ be the set of edges whose corresponding entries of $b$ have value $-1$ and $D_2$ be the set of edges whose corresponding entries of $b$ have value $1$. Let $b_{D_1, D_2}$ be the vector in $\{-1,0,1\}^N$ that corresponds to the move $(D_1, D_2)$.  If $\mathcal{MB}$ is a Markov basis for $A$, then $\mathcal {MB}_{sf}:=\mathcal {MB} \cap \{-1,0,1\}^{N}$ connects $\mathcal T_R$ (due to the fact that every tiling in $\mathcal T_R$ can be viewed as a $0-1$ vector and thus applying a move not in $\{-1,0,1\}^{N}$ would move us outside of the fiber $\mathcal T_R$).  

The Fundamental Theorem of Markov Bases gives a way to test whether or not a set $\mathcal{MB}$ is indeed a Markov basis.

\begin{thm} \cite{diaconis}
Let $\mathcal {MB}=\{b_1,\dots,b_n\} \subset \mathbb Z^N$ be a set of vectors; note that every vector $b_i$ can be written uniquely as the difference $b_i = b_i^{+}-b_i^{-}$ of two non-negative vectors with disjoint support. The set $\mathcal {MB}=\{b_1,\dots,b_n\}$ is a Markov basis for the matrix $A$ if and only if the corresponding set of binomials 
$
	\{ x^{b_i^{+}} - x^{b_i^{-}} \}_{i=1,\dots,n}
	% $b_i^{+}}  = \min(0,b_i)$ and ${b_i^{-}}=\max(0,b_i)$. 
$
generates the toric ideal $I_A$ \label{thm:FTMB}. 
\end{thm}

%We are concerned with moving between two tilings $T_1$ and $T_2$ of a cubiculated region. In essence, a move is replacing a set of dominoes $D_1$ with another set of dominoes $D_2$, such that the replacement results in a tiling. We will encode a move by $(D_1, D_2)$. A move has size $d$ if $|D_1|=|D_2|=d$.

\noindent Recall that a move is size $d$ if $|D_1|=|D_2|=d$.

\begin{thm}\label{thm:degreebound} If $I_{G_R}$ is generated by binomials of degree $d$ or less, then the set of tilings of a cubiculated region $R$ is connected by moves of size $d$ or less.

\end{thm}

\begin{proof}
Let  $A$ be the vertex-edge incidence matrix of $G_R$. Assume $I_{G_R}$ is generated by binomials of degree $d$ or less. Then  by Theorem \ref{thm:FTMB}, this means that there is a Markov basis $\mathcal{MB}$ for $A$ whose moves all have size $d$ or less. Since a Markov basis connects every fiber of $A$, and $\mathcal T_{R}$ is a fiber of $A$, the set $\mathcal {MB}$ connects $\mathcal T_{R}$. 
\end{proof}

 %We can then construct a set $\mathcal{C}$ where 
 %$$
 %\mathcal{C} = \{ %(u_1, v_1), 
 %$$

%such that every $g_i$ is of the form 
%$$
%g_i = \prod_{k=0}^i p_k - \prod_{k=0}^i q_k
%$$
%the binomial assoc with tiling  $\leftrightarrow$

%This result can be framed entirely in the language or toric ideals.

%\begin{thm}
%For all $2$D contractible regions; all binomials of degree $n$ (where $n$ is the number of squares) of a particular form (arising from two tilings) are generated by quadratics.
%\end{thm}

%\begin{proof}
%Consider the set of binomials, $B = x^{T_1}-x^{T_2}$ such that $deg(x^{T_1}) = deg(x^{T_2}) = n$. Then $B$ is of the following form
%$$
%B = \prod_{i \in T_1} x^{e_i} - \prod_{j \in T_2} x^{e_j}.
%$$
%We want to show that $I_B$ is generated by quadratics. Consider $\phi_{T_1,T_2} (B) $
%$$
%= \phi_{T_1, T_2} \Big(\prod_{i \in T_1} x^{e_i} - \prod_{j \in T_2} x^{e_j} \Big) = \phi_{T_1,T_2} \Big( \prod_{i \in T_1} x^{e_i}) - \phi_{T_1, T_2} \Big(\prod_{j \in T_2} x^{e_j} \Big)
%$$
%Then by distributing over the product we get
%$$
%= \prod_{i \in T_1} \phi_{T_1, T_2} (x^{e_i}) - \prod_{j \in T_2} \phi_{T_1, T_2}(x^{e_j})
%$$
%where $i \neq j$. We can see that 
%$$
%\phi_{T_1, T_2} ( x^{e_1}) = y_1y_2
%$$
%and 
%$$
%\phi_{T_1, T_2} (x^{e_2}) = y_2y_3.
%$$

%By expanding the product we get an alternating product that gives us
%$$
%\prod_{i \in w^+} y_i -\prod_{j \in w^-} y_j
%$$
%which is equivalent to 
%$$
%\prod_{i=0}^n \big( y_{2i} \prod_- y_{2i-1}\big)
%$$
%which is quadratic. 
%\end{proof}

\begin{thm} (Moves that connect tilings spaces) \label{thm:generators}
Let $R$ be a cubiculated region with the associated graph $G_R$. Let $\mathcal M_{cycles}$ be the set of moves on $\mathcal T_R$ corresponding to the set of chordless cycles of $G_R$, i.e. $$\mathcal M_{cycles} := \{b_{D_1, D_2} \ : \  D_1 \cup D_2 \text{ is a chordless cycle of  } G_R\}. $$
Then $\mathcal M_{cycles}$ connects $\mathcal T_R$. 
\end{thm}

\begin{proof} By Lemma 3.1 and 3.2 in \cite{oh}, since $I_{G_R}$ is bipartite, $I_{G_R}$ is generated by binomials arising from chordless cycles of $G_R$.  This means, by Theorem \ref{thm:FTMB}, $\mathcal M_{cycles}$ is a Markov basis for the vertex-edge incidence matrix of $G_{R}$, thus, $\mathcal M_{cycles} \subset \{-1, 0, 1\}^N$ connects $\mathcal T_{R}$.
\end{proof}

\begin{cor} \label{cor:movesdegreebound} 
Let $2d$ be the size of the largest chordless cycle in $G_R$. Then the space of tilings of $R$ is connected by moves of size $d$ or less.
\end{cor}

Note that Corollary \ref{cor:movesdegreebound} holds for any $R$ in any dimension. However, in many instances, especially in the 2-dimensional setting, the bound given in Corollary \ref{cor:movesdegreebound} is far from sharp.  This is due to the fact that tiling binomials can usually be written without using the larger degree generators of $I_G$, as we will see in Section \ref{sec:2D}.
%{\color{magenta} Let's add some examples here.  A 2-D example and some 3-D examples}\\

\begin{example}
$\ $

{\bf (a)} The graph $G_R$ of the cubiculated region $R$ in row (a) of Figure  \ref{fig:largestcycles} has a single chordless cycle of length $10$.  This means $I_{G_R}$ is a principal ideal generated by a binomial of degree $5$. The region $R$ has exactly two tilings that are connected by the tiling move of size $5$ that corresponds to the generating binomial of $I_{G_R}$. 

{\bf(b)} The largest cycle of the graph $G_R$ of the cubiculated region $R$ in row (b) of Figure  \ref{fig:largestcycles} is length $8$.  However, this length $8$ cycle is not chordless.  In fact, $G_R$ has no chordless cycles of length $>4$.  This means $I_{G_R}$ is generated by quadratics and $\mathcal T_R$ is flip connected.  Indeed, by this same reasoning and applying Theorem \ref{thm:generators}, we can conclude if $R$ is the $2 \times \ell$ box $B_{2, \ell}$, then $\mathcal T_R$ is flip connected. 

{\bf (c)} The graph $G_R$ of the cubiculated region $R$ in row (c) of Figure  \ref{fig:largestcycles} has a chordless cycle of length $10$.  The degree $5$ binomial in $I_G$ corresponding to this cycle is an \emph{indispensable binomial} of $I_G$ \cite{reyes}, meaning that there exists a nonzero constant multiple of it in every minimal system of binomial generators of $I_G$. However, unlike the region in row (a), for this region, the space of tilings $\mathcal T_R$ is flip connected and we do not need the size $5$ move to connect the space of tilings.  

{\bf (d)} The graph $G_R$ of the cubiculated region R in row (d) of Figure  \ref{fig:largestcycles} has a chordless cycle of length $6$, which corresponds to the \emph{trit} move described in \cite{milet}. The cubic binomial in $I_G$ corresponding to this cycle is an indispensable binomial of $I_G$. However, for this region, the space of tilings $\mathcal T_R$ is flip connected and we do not need the trit move; in fact, for $R$, there is no tiling for which we can apply the trit move.  

\begin{figure}[h]
\centering
\includegraphics[scale=.35]{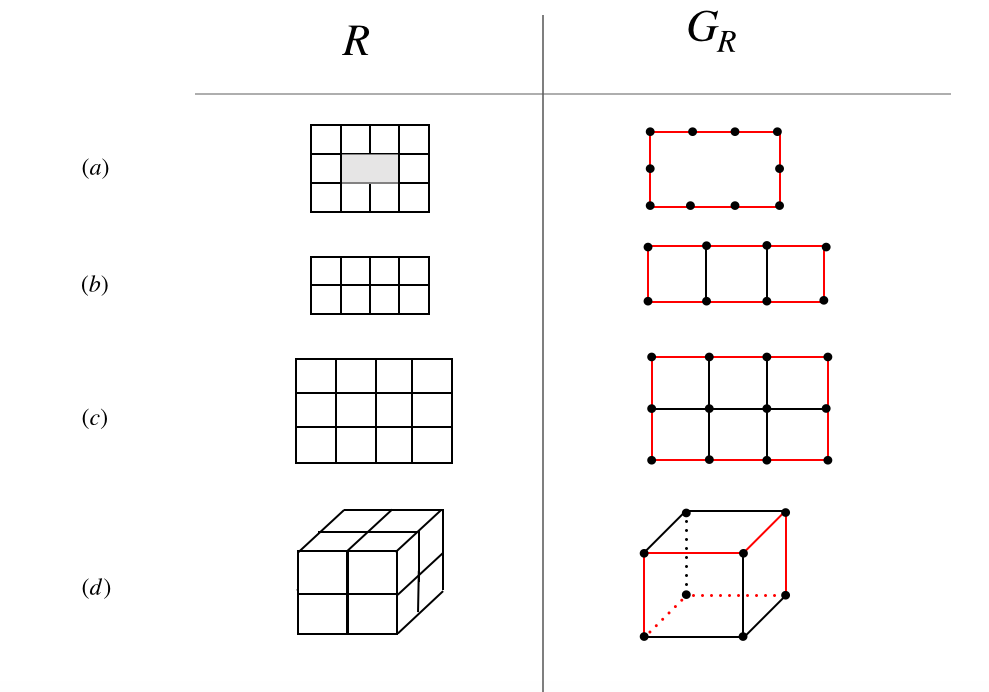}
\caption{Four cubiculated regions and their corresponding graphs with the largest cycle highlighted. }
\label{fig:largestcycles}
\end{figure}
\end{example}

\subsection{Tiling and flip ideals}

Theorem \ref{thm:generators} and Corollary \ref{cor:movesdegreebound} give a bound on the size of moves needed to connect the space of tilings of a region $R$, however, this bound can be arbitrarily large.  For example, let $m, n \geq 3$, then $G_{m,n}$ contains a chordless cycle of length $2(m+n)$.

A local flip $(D_1, D_2)$ corresponds to a $4$-cycle in $G_R$ and the corresponding binomial $y^{D_1} - y^{D_2}$ has degree $2$. Conversely, any non-zero quadratic binomial in $I_{G_R}$ must correspond to a $4$-cycle, and consequently, a flip move. Thus, to show $\mathcal T_R$ is flip connected, we need to show that every binomial arising from two tilings is generated by quadratics. 

\begin{defn}
Let $R$ be a cubiculated region with associated graph $G_R$.  The \emph{flip ideal} of $R$ is defined as follows:

$$I_{R_{flip}} := \langle y^{D_1}- y^{D_2} \ : \ (D_1, D_2) \in  \mathcal M_{flip}  \rangle \subseteq I_{G_R}.$$
The \emph{tiling ideal} of $R$ is defined as follows:

$$I_{R_{tiling}} := \langle y^{T_1} - y^{T_2} \ : T_1,\ T_2 \in \mathcal T_{R} \  \rangle \subseteq I_{G_R}.
$$

\end{defn}

Using the flip and tiling ideal, we can use the language of ideals to restate what it means for a region to be flip connected.

\begin{thm} \label{thm:flipconnected}
A tiling space $\mathcal T_r$ of a cubiculated region $R$ is flip connected if and only if 
$$I_{R_{tiling}} \subseteq I_{R_{flip}}.$$
\end{thm}

For a region $R$, both the flip ideal and the tiling ideal ideal are subideals of the toric ideal of the graph $G_R$:
$$I_{R_{tiling}} \subseteq I_{G_R}  \ \ \ \ \ \ \ \ \ \ \ \ \ \ \ \ I_{R_{flip}} \subseteq I_{G_R}.$$

\noindent When $G_R$ contains no chordless cycles of length $>4$, we have $I_{R_{flip}} = I_{G_R}$, and thus, $I_{R_{tiling}} \subseteq I_{flip}$ and $R$ is flip connected.

While $I_{R_{tiling}}$ and $I_{R_{flip}}$ are binomial ideals, unlike $I_{G_R}$ they are not always prime ideals and therefore not always toric ideals.  However, the primary decomposition of the flip ideal of a region has interesting combinatorics.  Indeed, working from an earlier version of this manuscript, in \cite{chin}, Chin explores the flip ideals of $3 \times \ell$ box regions and gives a complete description of their primary decompositions.

We now explore the three ideals $I_{G_R}$, $I_{R_{tiling}}$, and $I_{R_{flip}}$ and their possible relationships through four examples. 
\begin{example} 
Let $R = B_{2,3}$, the $2 \times 3$ box.
\begin{figure}[h]
\begin{center}
\includegraphics[scale =.5]{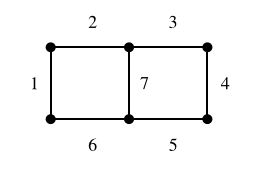}
\caption{The graph $G_{2,3}$ with edges labeled.}
\label{fig:2x3box}
\end{center}
\end{figure}
Using the labeling in Figure \ref{fig:2x3box}, we have
\begin{enumerate}
    \item[(1)] $I_{G_{R}} = \langle \; y_1y_7-y_2y_6, \; y_4y_7-y_3y_5 \;\rangle$, \\
    
    \item[(2)] $I_{R_{tiling}}= \langle \; y_1y_3y_5-y_2y_4y_6 ,\; y_1y_4y_7-y_1y_3y_5, \;y_1y_4y_7-y_2y_4y_6 \;\rangle$,\\
    
    \item[(3)] $I_{R_{flip}} = \langle \; y_4y_7-y_3y_5,\; y_1y_7-y_2y_6 \;\rangle$.\\
\end{enumerate}
\begin{figure}[H]
\centering

\includegraphics[scale=.35]{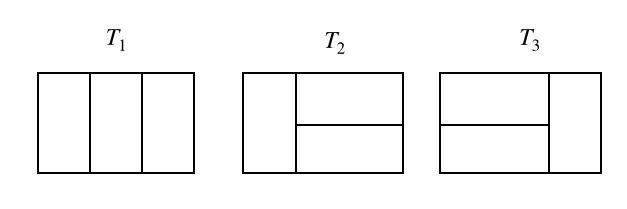} \\
\includegraphics[scale=.45]{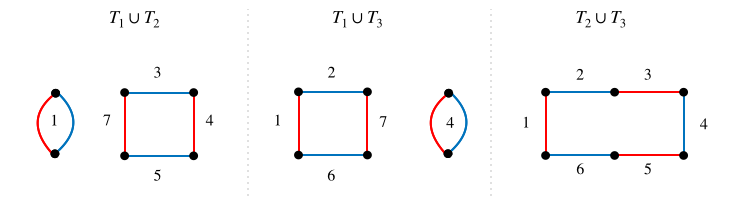}
\caption{On top are the three tilings, $T_1, T_2$, and $T_3$, of the $2 \times 3$ box.   On bottom are the three cycle covers of $G_{2,3}$ formed by the possible pairs of tilings.}
\label{fig:grid}
\end{figure}

The three tilings of $R$ are depicted in Figure \ref{fig:grid}.  Notice that the last binomial listed in the generating set of $I_{R_{tiling}}$ is not needed and we could have written $I_{R_{tiling}}= \langle  y_1y_3y_5-y_2y_4y_6 ,  y_1y_4y_7-y_1y_3y_5 \rangle$.  For this region, it is the case that
$$
I_{R_{tiling}} \subseteq I_{R_{flip}} = I_{{G}_{R}}.
$$
and thus $\mathcal T_R$ is flip connected by Theorem \ref{thm:flipconnected}.
\end{example}

\begin{example}
Now let $R$ be the $2 \times 2 \times 2 $  box. The associated grid graph with edge labels is depicted in Figure 6. 
\begin{figure}[H]
\begin{center}
\includegraphics[scale=.4]{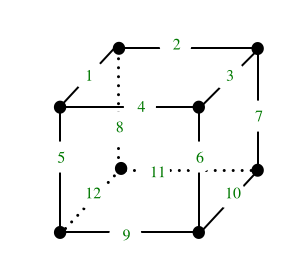}
\caption{The associated grid graph of the $2 \times 2 \times 2$ box.}
\end{center}
\label{fig:labcube}
\end{figure}
\noindent We compute the following
\begin{enumerate}

    \item[(1)] $I_{G_{R}} =  \langle  \;y_1y_3-y_2y_4, \; y_3y_{10} - y_6y_7, \; y_4y_9-y_5y_6, \; y_1y_{12}-y_5y_8, \; y_2y_{11}-y_7y_8, \; y_9y_{11}-y_{10}y_{12}, \; y_2y_5y_{10}-y_1y_9y_7, \;  y_4y_7y_{12}-y_3y_{5}y_{11}, \; y_2y_6y_{12}-y_3y_8y_9, \; y_1y_6y_{11}-y_4y_{8}y_{10} \; \rangle$\\
    
    \item[(2)]$I_{R_{tiling}} = \langle \; y_2y_5y_6y_{11}-y_5y_6y_7y_{8}, \;
    y_4y_9y_7y_{8}-y_5y_6y_7y_{8}, \;
    y_2y_4y_9y_{11}-y_4y_9y_7y_{8}, \;
    y_2y_4y_9y_{11}-y_2y_4y_{10}y_{12}, \;
     y_1y_3y_{10}y_{12}-y_2y_4y_{10}y_{12}, \;
     y_1y_3y_{10}y_{12}-y_1y_3y_{9}y_{11}, \;
     y_1y_3y_{10}y_{12}-y_1y_6y_{11}y_{12}, \;
     y_1y_3y_{10}y_{12}-y_3y_5y_{8}y_{10}
    \rangle$\\
    
    \item[(3)]$I_{R_{flip}} = \langle \; y_1y_3-y_2y_4, \; y_3y_{10} - y_6y_7, \; y_4y_9-y_5y_6, \; y_1y_{12}-y_5y_8, \; y_2y_{11}-y_7y_8, \; y_9y_{11}-y_{10}y_{12} \; \rangle. $
\end{enumerate}

For this example, while $I_{R_{flip}} \neq I_{{G}_{R}}$, we do have that $I_{R_{tiling}} \subseteq I_{R_{flip}}$.  This can be seen by noticing that every generator of $I_{R_{tiling}}$ is a monomial multiple of an element in $I_{R_{flip}}$. Since $I_{R_{tiling}} \subseteq I_{R_{flip}}$, the tiling space $\mathcal T_R$ is flip connected by Theorem \ref{thm:flipconnected}.
\end{example}

\begin{example}
In this example, let  $R = B_ {3,4}$, the $3 \times 4$  box.  The associated grid graph with edge labels is depicted in Figure \ref{fig:3x4box}.
\begin{figure}[H]
    \centering
    \includegraphics[scale=.35]{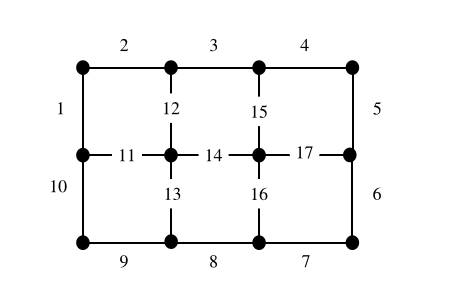}
    \caption{The grid graph associated with the $3 \times 4 $ box with edges labeled $1$ through $17$.}
    \label{fig:3x4box}
\end{figure}
\noindent We compute the following
\begin{enumerate}
    \item[(1)] $I_{G_{R}} = \langle \; 
    y_1y_{12}-y_2y_{11}, \; y_{12}y_{15}-y_3y_{14}, \; y_5y_{15}-y_4y_{17}, \; y_{10}y_{13}-y_9y_{11}, \; y_{13}y_{16}-y_8y_{14}, \; y_6y_{16}-y
    _7y_{17},\;
    y_1y_3y_{9}y_{16}-y_2y_{8}y_{10}y_{15}, \; y_4y_6y_8y_{12}- y_3y_5y_7y_{13}, \;
    y_1y_3y_5y_7y_9- y_2y_4y_6y_8y_{10} \;   \rangle$\\
    \item[(2)] $I_{R_{tiling}} = \langle \;  y_1y_{5}y_{7}y_9y_{12}y_{15}-y_2y_{4}y_7y_{9}y_{11}y_{17}, \;  y_1y_{5}y_{7}y_9y_{12}y_{15} - y_2y_{4}y_6y_8y_{10}y_{14}, \; y_1y_{5}y_{7}y_9y_{12}y_{15} - y_1y_{3}y_5y_7y_9y_{14}, \; y_1y_{5}y_{7}y_9y_{12}y_{15}-y_2y_{4}y_{6}y_{10}y_{13}y_{16}, \;   y_2y_5y_9y_7y_{11}y_{15}-y_1y_3y_{5}y_7y_9y_{14}, \;
    y_2y_5y_9y_7y_{11}y_{15} - y_2y_4y_6y_9y_{11}y_{16}, \;
    y_2y_4y_6y_9y_{11}y_{16} - y_1y_4y_7y_9y_{12}y_{17}, \;
    y_1y_4y_7y_9y_{12}y_{17} -
    y_1y_4y_6y_9y_{12}y_{16}, \;
    y_1y_4y_6y_9y_{12}y_{16} - 
    y_2y_4y_7y_{10}y_{13}y_{17}, \;
    y_2y_4y_7y_{10}y_{13}y_{17} -
    y_2y_5y_7y_{10}y_{13}y_{15}
    \rangle$\\
    
    \item[(3)] $I_{R_{flip}} = \langle \; y_1y_{12}-y_2y_{11}, \; y_{12}y_{15}-y_3y_{14}, \; y_5y_{15}-y_4y_{17}, \; y_{10}y_{13}-y_9y_{11}, \; y_{13}y_{16}-y_8y_{14}, \; y_6y_{16}-y
    _7y_{17}\; \rangle$.\\
\end{enumerate}
\end{example}
In this example, as with the previous example, $I_{R_{flip}} \neq I_{{G}_{R}}$, but  $I_{R_{tiling}} \subseteq I_{R_{flip}}$, hence, the tiling space $\mathcal T_R$ is flip connected by Theorem \ref{thm:flipconnected}.
\begin{example}
Let $R$ be the 3-dimensional region whose associated graph is pictured in Figure \ref{fig:subetrit}. 
\begin{figure}[H]
    \centering
    \includegraphics[scale=.4]{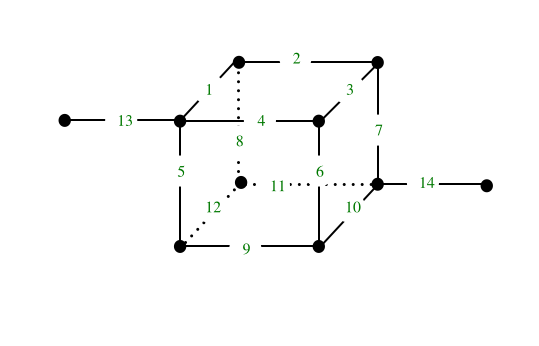}
    \caption{The labeled graph $G_R$ of a $3$-dimensional region $R$ that contains the $2 \times 2 \times 2$ cube as a subregion.}
    \label{fig:subetrit}
\end{figure}
\noindent We compute the following\\

\begin{enumerate}
    \item[(1)] $I_{G_{R}} =  \langle  \;y_1y_3-y_2y_4, \; y_3y_{10} - y_6y_7, \; y_4y_9-y_5y_6, \; y_1y_{12}-y_5y_8, \; y_2y_{11}-y_7y_8, \; y_9y_{11}-y_{10}y_{12}, \; y_2y_5y_{10}-y_1y_9y_7, \;  y_4y_7y_{12}-y_3y_{5}y_{11}, \; y_2y_6y_{12}-y_3y_8y_9, \; y_1y_6y_{11}-y_4y_{8}y_{10} \; \rangle$\\
    
    \item[(2)] $I_{R_{tiling}} = \langle  \; y_2y_6y_{12}y_{13}y_{14}-y_3y_9y_8y_{13}y_{14} \; \rangle$ \\
    
    \item[(3)] $I_{R_{flip}} = \langle \; y_1y_3-y_2y_4, \; y_3y_{10} - y_6y_7, \; y_4y_9-y_5y_6, \; y_1y_{12}-y_5y_8, \; y_2y_{11}-y_7y_8, \; y_9y_{11}-y_{10}y_{12} \; \rangle. $\\
\end{enumerate}

In this example, there are only two tilings of $R$ and thus $I_{R_{tiling}}$ is a principal ideal. The tiling ideal is not contained in the flip ideal, and hence, the tiling space is not flip connected.  Indeed, a trit move is needed to connect the space, which can be seen by noting that the single generator of $I_{R_{tiling}}$ can be factored into a monomial and cubic trit binomial 
$$y_{13}y_{14}(y_2y_6y_{12}-y_3y_9y_8).$$
\end{example}

\subsection{Tiling binomials in terms of cycle covers}

Recall that for two tilings $T_1$ and $T_2$ of $R$, their union $T_1 \cup T_2$ is a cycle cover of $G_R$. In this section, we state and prove a lemma regarding such cycle covers that will be helpful in giving an algebraic proof of Theorem \ref{thm:Thurston}.

Let $T_1$ and $T_2$ be two tilings of a cubiculated region $R$ with corresponding cycle cover $T_1 \cup T_2 = \mathcal C=\{\mathcal C_1, \ldots, \mathcal C_r\}$ of $G_R$.  We define the cycle binomial $B_{\mathcal C_i}$ corresponding to the cycle $\mathcal C_i$ as follows. Construct a closed walk $w_i$ on each $\mathcal C_i \in \mathcal C$ by starting with an edge in $T_1$ and then walking in either direction. Then the cycle binomial $B_{\mathcal C_i}$ is the binomial arising from the walk $w_i$:

$$B_{\mathcal C_i}:= B_{w_i}=\prod_{e \in T_1 \cap \mathcal C_i} y_e - \prod_{e \in T_2 \cap \mathcal C_i} y_e.$$ 

\begin{lemma} \label{lem:cycles} Let $T_1$ and $T_2$ be two tilings of a cubiculated region $R$ with corresponding cycle cover $T_1 \cup T_2 = \mathcal C=\{\mathcal C_1, \ldots, \mathcal C_r\}$ of $G_R$. Then $B_{T_1,T_2}$ can be written as the sum of $r$ binomials where the $i$th binomial can be factored into a monomial and the cycle binomial $B_{\mathcal C_i}$.
\end{lemma}
\begin{proof} We begin by describing the $i$th monomial that appears in the sum described by the lemma. Let
$$m_1  = y^{T_1 \setminus \mathcal C_1},$$
and for $i = 2, \ldots, r$, let 
$$m_i = y^{(T_1 \setminus (\mathcal C_1 \cup  \cdots \cup \mathcal C_i)) \bigcup (T_2 \cap (\mathcal C_1 \cup \cdots \cup \mathcal C_{i-1}))}.$$
Then, we can write $B_{T_1, T_2}$ as
$$B_{T_1, T_2} = m_1B_{\mathcal C_1} + m_2B_{\mathcal C_2} + \cdots + m_rB_{\mathcal C_r}.$$

\end{proof}

\begin{remark} Note that the  binomial arising from a $2$-cycle $\mathcal C_i$ has the following form $$y_j - y_j.$$
Therefore adding $B_{\mathcal C_i}$ is equivalent to adding zero.  Thus, we can obtain a similar statement to Lemma \ref{lem:cycles} by letting $\mathcal C$ be all cycles of $T_1 \cup T_2$ of length greater than 2. 
\end{remark}

\begin{example}\label{twoperfmatch}

Let $R$ be the $2 \times 5$ box and consider the two tilings $T_1$ and $T_2$ pictured in Figure \ref{fig:twoperfmatch}. The cycle cover $\mathcal C = \{ \mathcal C_1, \mathcal C_2\}$ of $G_R$ induced by $ T_1 \cup T_2$ is also shown in Figure \ref{fig:twoperfmatch}.

\begin{figure}[h]
\centering
\includegraphics[scale=.4]{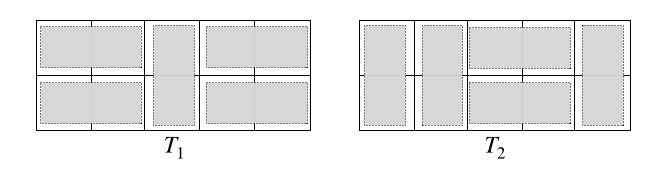} \vspace{2pt} \includegraphics[scale = .4]{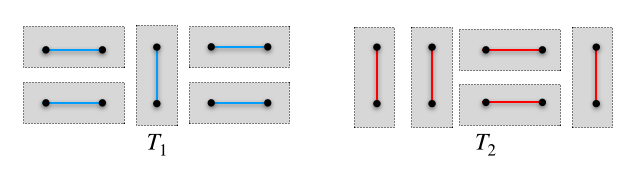}
\vspace{2pt}
\includegraphics[scale=.5]{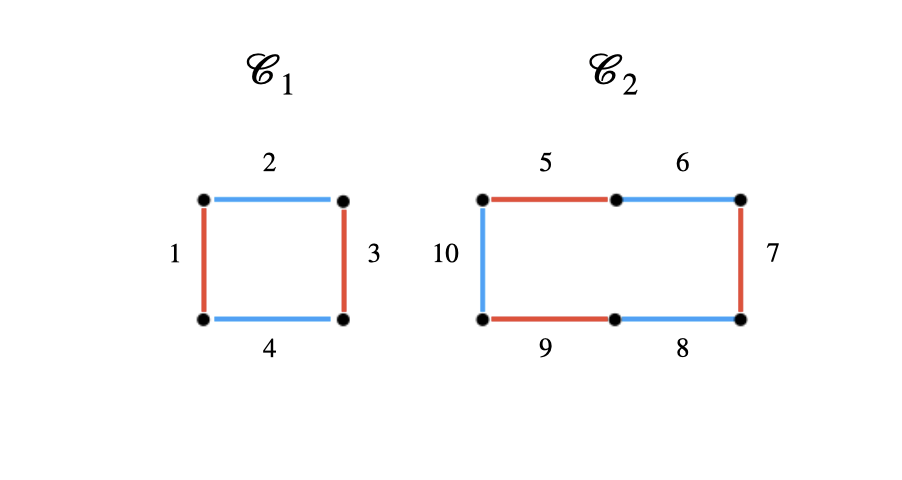}
\caption{The top row shows two tilings, $T_1$ and $T_2$, of the $2 \times 5$ box.  The bottom row shows the two cycles formed when taking the union of $T_1$ and $T_2$ as edge sets of $G_R$. }
\label{fig:twoperfmatch}
\end{figure}
  %\begin{minipage}[b]{0.4\textwidth}
   % \includegraphics[width=\textwidth]{tiling2.png}
    %\caption{Tiling, $T_2$}
  %\end{minipage}
The binomial arising from $(T_1, T_2)$ is
$$B_{T_1, T_2} = y_1y_3y_5y_7y_9 - y_2y_4y_6y_8y_{10}.
$$
Using Lemma \ref{lem:cycles}, we can write $B_{T_1, T_2}$ in terms of $B_{\mathcal C_1}=y_1y_3-y_2y_4$ and $B_{\mathcal C_2}=y_5y_7y_9-y_6y_8y_{10}$ as follows

$$
B_{T_1, T_2} = y_5y_7y_9(y_1y_3-y_2y_4) + y_2y_4(y_5y_7y_9-y_6y_8y_{10}).
$$
Notice that the factored monomials in this sum have the form described in the proof of Lemma \ref{lem:cycles}.  In particular, working on the labels of the indeterminates appearing in each monomial, we have

$$\{5, 7, 9\} = T_1 \setminus \mathcal C_1 =  \{1, 3, 5, 7, 9\} \setminus \{1,2,3,4\},$$

and

$$\{2, 4\} = (T_1 \setminus (\mathcal C_1 \cup C_2)) \cup (T_2 \cap \mathcal C_1) = \emptyset \cup  (T_2 \cap  \mathcal C_1)
=\{2, 4, 6, 8, 10\} \cap \{1,2,3,4\}. $$

\end{example}

\section{Connectivity of Tilings of $2$-Dimensional Regions } \label{sec:2D}

We conclude this paper by showing any binomial arising from two tilings of a $2$-dimensional simply connected cubiculated region $R$ is generated by quadratics.  This allows us to prove that the space of tilings of $R$ is flip connected.

\begin{defn}
A region $R$ is said to be \emph{simply connected} if any simple closed curve can be shrunk to a point continuously in the set.

\begin{remark}
A $2$-dimensional region is simply connected if it  has no holes.
\end{remark}
\end{defn}

%{\color{magenta} Let's add definition of simply connected.}

\begin{thm}\label{tilemain} Let $R$ be a $2$-dimensional simply connected cubiculated region. Any binomial arising from two distinct tilings of $R$ is generated by quadratics. In particular, $I_{R_{tiling}} \subseteq I_{R_{flip}}.$
\end{thm}

Our proof of Theorem \ref{tilemain} relies on a couple of lemmas that we will prove first.

\begin{defn}
Let $R$ be a $2$-dimensional simply connected cubiculated region with graph $G_R$.  We call a cycle $\mathcal C_1$ of $G_R$ a \emph{contractible cycle} if, when $G_R$ is drawn in the plane as a grid graph, the interior of $\mathcal C_1$ contains no vertices.
\end{defn}

\noindent Contractible cycles can be regarded as Hamiltonian cycles of a subgraph $G_{R'}$ of $G_{R}$ where $R'$ is a simply connected subregion of $R$.  Indeed, if $R$ is simply connected, then any Hamiltonian cycle of $G_{R}$ is a contractible cycle.  Moreover, when $R$ is a $2$-dimensional simply connected region, any $2$-cycle of $G_R$ is a contractible cycle.

\begin{defn}
Let $R$ be a $2$-dimensional simply connected cubiculated region with graph $G_R$.  We call a cycle $\mathcal C_0$ of $G_R$ a \emph{perimeter cycle} if, when $G_R$ is drawn in the plane as a grid graph, the interior of $\mathcal C_0$ contains no cycles or only $2$-cycles. (The name \emph{perimeter} comes from that fact that if $T_1 \cup T_2$ contains a single perimeter cycle and no other cycles besides $2$-cycles, then the tilings $T_1$ and $T_2$ only differ on the perimeter of a simply connected subregion.)
\end{defn}

Note that a contractible cycle is a perimeter cycle, and thus, a $2$-cycle is a perimeter cycle.

\begin{lemma}\label{lem:perimetertocontractible}
Let $T_1$ and $T_2$ be two tilings of a simply connected $2$-dimensional cubiculated region $R$ such that $T_1 \cup T_2$ is a collection of perimeter cycles $\mathcal C_1, \ldots, \mathcal C_r$.  There exists a sequence of flip moves $(D_{1_1}, D_{2_1}), \ldots, (D_{1_s}, D_{2_s})$ that takes $T_1$ to $T_1'$ and a sequence of flip moves $(D'_{1_1}, D'_{2_1}), \ldots, (D'_{1_t}, D'_{2_t})$ that takes $T_2$ to $T_2'$ such that $T_1' \cup T_2'$ is a collection of contractible cycles. In particular, 
 \begin{align*}
 B_{T_1, T_2} &= y^{T_1 \setminus D_{1_1}}(y^{D_{1_1}} - y^{D_{2_1}}) + \ldots + y^{T_1' \setminus D_{2_s}}(y^{D_{1_s}} - y^{D_{2_s}}) \\ &+ B_{T_1', T_2'}\\ &+ y^{T_2' \setminus D'_{2_t}}(y^{D'_{2_t}} - y^{D'_{1_t}}) + \ldots  + y^{T_2 \setminus D'_{1_1}}(y^{D'_{2_1}} - y^{D'_{1_1}})
 \end{align*}
where each binomial of the form $y^{D_{1_i}} - y^{D_{2_i}}$ or $(y^{D_{2_i}} - y^{D_{1_i}})$ has degree $2$.
 \end{lemma}

\begin{proof}
Let $k$ be the number of vertices contained in the interiors of $\mathcal C_1, \ldots, \mathcal C_r$ when $G_R$ is drawn in the plane as a grid graph.  We will induct on $k$.

For the base case, assume $k = 0$. Then $\mathcal C_1, \ldots, \mathcal C_r$ are all contractible cycles. Therefore, $T_1 \cup T_2$ is a collection of contractible cycles. 

Now suppose the statement is true for up to $k-1$ internal vertices and assume the cycles in $\mathcal C$ have $k>0$ internal vertices. Since $k>0$, the interior of at least one of $\mathcal C_1, \ldots, \mathcal C_r$ contains at least one 2-cycle, let's assume $\mathcal C_1$ contains at least one 2-cycle. We will consider three cases based on the positions of the interior $2$-cycles.\\
\\
\noindent \begin{it} 
Case 1: An interior $2$-cycle is parallel to $\mathcal{C}_1$. 
\end{it}

\begin{figure}[h]
\begin{center}
    \includegraphics[scale = .35]{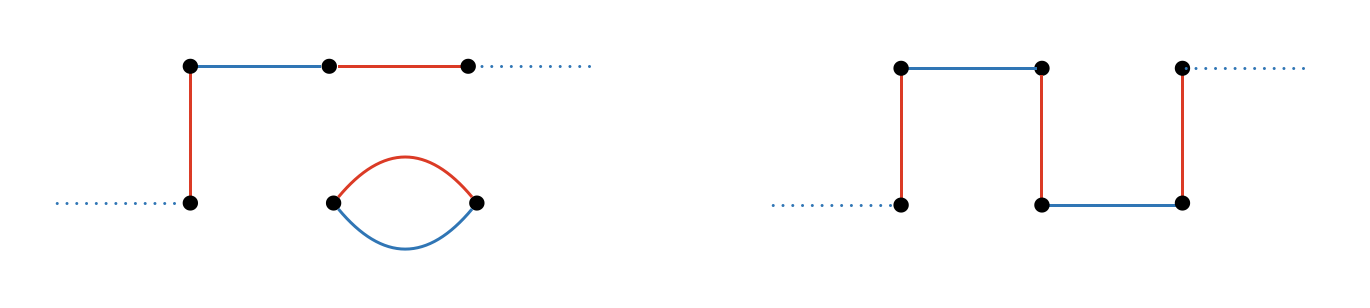}
    \caption{A local flip is  performed on with the  $2$-cycle parallel to $\mathcal{C}_1$.} 
    \label{fig:Case1}
\end{center}
\end{figure}

In this case we can perform a local flip the edge parallel to the $2$-cycle as shown in Figure \ref{fig:Case1}. This yields a decrease in the number of internal vertices by $2$, and then, we can apply the induction hypothesis.\\ 
\\
\begin{it}
Case $2$:  There exists an interior $2$-cycle that is not parallel to $\mathcal{C}_1$. 
\end{it}

%\begin{figure}[h]
%\begin{center}
%    \includegraphics[scale= .35]{2cycnotparallel.png} 
%    \includegraphics[scale=.35]{swaptooutside.png}
%    \caption{A simple swap performed on the parallel red edges resulting in a $2$-cycle outside of $\mathcal{C}_0$.} 
%\end{center}
%\end{figure}
\begin{figure}
    \centering
    \includegraphics[scale=.30]{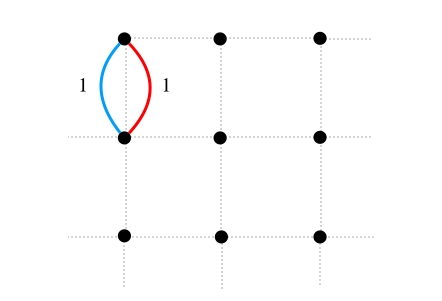}
    \includegraphics[scale=.30]{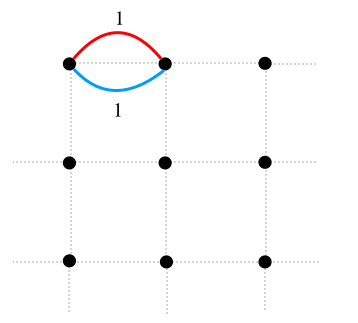} 
    \caption{On the left is a \emph{north-south} 2-cycle.  On the right is an \emph{east-west} $2$-cycle.}
    \label{fig:2cycles}
\end{figure}

$\ $ \\
Let us categorize $2$-cycles into two types: north-south cycles and east-west cycles (see Figure \ref{fig:2cycles}).  In the case where there is no $2$-cycle parallel to $\mathcal C_1$, then (i) there exists a east-west cycle such that the two vertices in $G_R$ directly north of the $2$-cycle or the two vertices in $G_R$ directly south of the $2$-cycle are both contained in $\mathcal C_1$ or (ii) there exists a north-south cycle such that the two vertices in $G_R$ directly east of the $2$-cycle or the two vertices in $G_R$ directly west of the $2$-cycle are both contained in $\mathcal C_1$. We note that if either $(i)$ or $(ii)$ does not hold, then starting at any interior $2$-cycle, there is an infinite sequence of $2$-cycles that can be constructed by choosing the $2$-cycle that covers at least one vertex to the north or east of the previous cycle, and thus $R$ is not finite.

Without loss of generality, let's assume that there is a $2$-cycle $\mathcal C_2$ of the form described in situation (i) such that the two vertices directly north of the $\mathcal C_2$ are both contained in $G_R$. In this situation, $\mathcal C_1$ must transverse the vertices north of $\mathcal C_2$ in the way illustrated in Figure \ref{fig:notparallel}.  Note that the edges $e_1$ and $e_2$ in Figure \ref{fig:notparallel} must be from the same tiling since every chord of $\mathcal C_1$ must be even. After performing a local flip, $\mathcal C_1$ is split into two cycles, $\mathcal C_1'$ that contains $\mathcal C_2$ and $\mathcal C_2'$ that contains no vertices or only $2$-cycles and thus is a new perimeter cycle; see Figure \ref{fig:notparallelafterflip}. We have not reduced the number of interior vertices, however, we now meet the conditions of Case 1 and can proceed accordingly.
%After performing the simple swap we can handle Case 2 is now reduced back to Case 1. Let the number of interior $2$-cycles be greater than or equal to $2$.

%\begin{bf}
%Case $3$: There is no $2$-cycle parallel to $\mathcal{C}_0$. 
%\end{bf}

%\begin{figure}[h]
%\begin{center}
%    \includegraphics[scale= .35]{side.png} 
%    \includegraphics[scale=.35]{flipperim.png} 
%    \includegraphics[scale=.35]{outter.png}
%    \caption{Side of graph}
%\end{center}
%\end{figure}

%Assuming there is no $2$-cycle that is parallel to $\mathcal{C}_0$. Without loss of generality it is sufficient to look at one side of the cubiculated region. By our hypothesis all interior vertices are covered with $2$-cycles. Specifically, $2$-cycles that are not parallel to $\mathcal{C}_0$. Each vertex that is parallel to $\mathcal{C}_0$ must be covered by a $2$-cycle that is perpendicular to $\mathcal{C}_0$. Hence, graph must be of the form depicted in Figure $3$.
%Performing a simple swap on any two parallel $2$-cycles leaves us with a $4$-cycle parallel to $\mathcal{C}_0$. We can connect to $\mathcal{C}_0$ by performing a flip move. Notice that this sequence of moves decreases the number of internal vertices by $2$. Again, we can apply our inductive hypothesis.
\end{proof}

\begin{figure}[h]
\centering
\includegraphics[scale=.5]{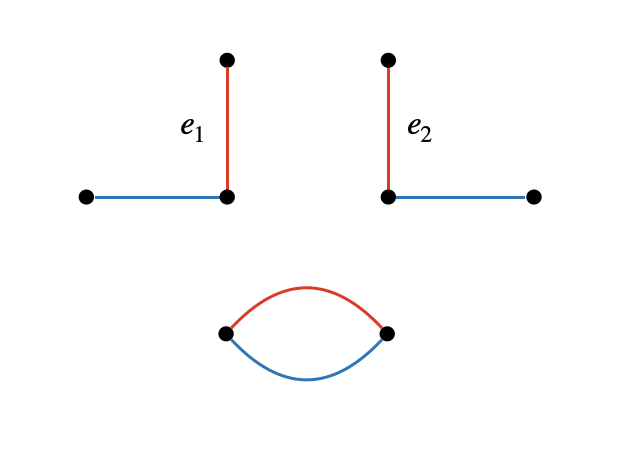}
\caption{An east-west cycle such that the two vertices in $G_R$ directly north of the $2$-cycle are traversed by $\mathcal C_1$.}
\label{fig:notparallel}
\end{figure}

\begin{figure}[h]
\centering
\includegraphics[scale=.5]{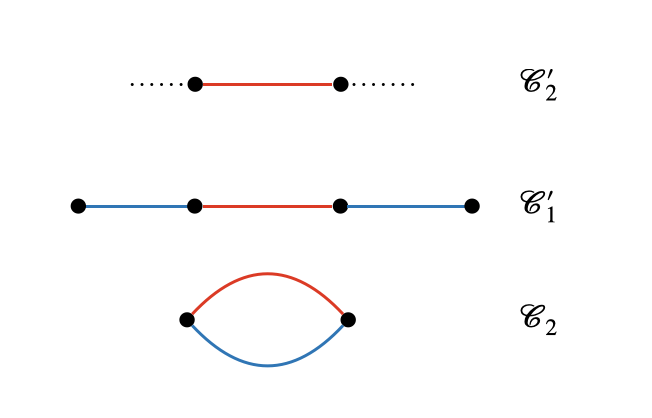}
\caption{Configuration from Figure \ref{fig:notparallel} after performing a local flip.}
\label{fig:notparallelafterflip}
\end{figure}

\begin{lemma} \label{lem:contractiblecycles} Let $R$ be a simply connected $2$-dimensional cubiculated region and let $T_1, T_2$ be two tilings such that $T_1 \cup T_2$ contains a single contractible cycle of $G_R$ of length $\geq 4$ and no other cycles besides $2$-cycles.  Then $B_{T_1, T_2}$ is generated by quadratics.
\end{lemma}

\begin{proof}
Let $\mathcal{C}_1$ be the single contractible cycle of length $\geq 4$.  We will proceed by induction on the length of $\mathcal{C}_1$, which we will denote by $k$.

For the base case, let $k=4$.  Then $\mathcal C_1$ is a single $4$-cycle.  Let $\mathcal C_1$ be labeled as in Figure \ref{fig:4cylewithlabels}, then $B_{T_1, T_2}$ is the product of a monomial and the quadratic $y_1y_3 - y_2y_4$, and thus is generated by quadratics.

\begin{figure}
    \centering
    \includegraphics[scale=.45]{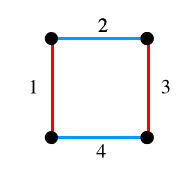}
    \caption{A single $4$-cycle.}
    \label{fig:4cylewithlabels}
\end{figure}

Now assume $\mathcal C_1$ has length $k>4$ and the statement holds whenever the length of $\mathcal C_1$ is less than $k$. Embed $G_R$ into a grid graph and let row $i$ be the first row (scanning from north to south) that contains a vertex covered by $\mathcal C_1$ and let column $j$ be the first column in row $i$ that contains a vertex covered by $\mathcal C_1$;  we will call the $(i,j)$th vertex of the grid graph, the north-west corner of $\mathcal C_1$ and refer to it as $v$. 

By the way we selected the north-west corner, the vertex $v$ must be traversed by $\mathcal C_1$ as illustrated in Figure \ref{fig:north-westcorner}.  Furthermore, since $R$ is simply connected, and $\mathcal C_1$ is a contractible cycle, the highlighted vertex in Figure \ref{fig:north-westcorner} must also be covered by $\mathcal C_1$.  The highlighted vertex can be covered in three ways as shown in Figure \ref{fig:threeways}.

\begin{figure}
    \centering
    \includegraphics[scale=.5]{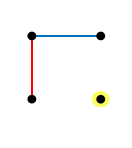}
    \caption{The north-west corner of $\mathcal C_1$.}
    \label{fig:north-westcorner}
\end{figure}

\begin{figure}
    \centering
    \includegraphics[scale=.5]{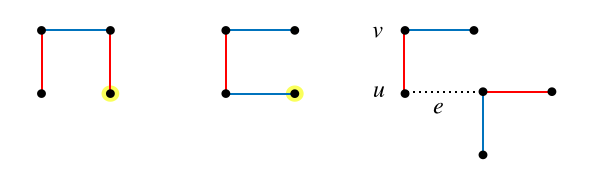}
    \caption{The three ways the north-west corner can be traversed by $\mathcal C_1$.}
    \label{fig:threeways}
\end{figure}

For the first two cases illustrated in Figure \ref{fig:threeways}, we can perform a local flip move that decomposes $C_1$ into a $2$-cycle and a contractible cycle of length less than $\mathcal C_1$, and then apply the induction hypothesis.

For the third case illustrated in Figure \ref{fig:threeways}, note that the edge $e$ in $G_R$ from $u$, the vertex south of $v$, to the highlighted vertex is an even chord of $\mathcal C_1$, since $G_R$ contains only even cycles.  Using $e$, we can split $\mathcal C_1$ into two even contractible cycles, $\mathcal C_2$ and $\mathcal C_3$, that overlap on the edge $e$.  Then, similar to the proof of Lemma \ref{lem:cycles}, and assuming the edges adjacent to $e$ in $\mathcal C_2$ belong to $T_1$, we can define the following two pairs of tilings:
\begin{align*}
S_1 & = T_1 \ \ \ \ \ &S_2 & = (T_1 \setminus \mathcal C_2) \cup (\mathcal C_2 \setminus T_1) \\ U_1 & = (T_1 \setminus \mathcal C_2) \cup (\mathcal C_2 \setminus T_1) = (T_2 \setminus \mathcal C_3) \cup (\mathcal C_3 \setminus T_2) \ \ \ \ \ \ &U_2 & =  T_2.  
\end{align*}
The binomial $B_{T_, T_2}$ can be written in terms of binomials arising from $S_1, S_2$ and $U_1, U_2$:
$$B_{T_1, T_2} = B_{S_1, S_2} + B_{U_1, U_2}.$$
Since $S_1 \cup S_2$ and $U_1 \cup U_2$ both contain only $2$-cycles and a single contractible cycle of length $\geq 4$ but less than $k$,  the binomials $B_{S_1, S_2}$ and $B_{U_1, U_2}$ are both generated by quadratics and thus $B_{T_1, T_2}$ is generated by quadratics.
\end{proof}

We now can begin our proof of Theorem \ref{tilemain}.  We will induct on the binomial degree of the tiling binomial.

\begin{defn}
Let $b$ be a homogeneous binomial of the form $b=m(u-v)$ where $m, u, v$ are monomials and $\gcd(u,v)=1$.  We will call the $\deg u = \deg v$, the \emph{binomial degree} of $b$.
\end{defn}

\begin{proof}[Proof of Thoerem \ref{tilemain}]  Let $B_{T_1, T_2}$ be a non-zero binomial arising from two tilings $T_1, T_2$. We will induct on the binomial degree of $B_{T_1, T_2}$.  

In the base case, let's assume the binomial degree of $B_{T_1,T_2}$ is 2.  Then 
$$B_{T_1, T_2} = y^{T_1 \setminus D_1}(y^{D_1} - y^{D_2})$$
where the size of the move $(D_1, D_2)$ is 2 and thus a flip move.

Now, assume that the binomial degree of $B_{T_1,T_2}$ is $k>2$ and the statement holds whenever the binomial degree is less than $k$.  Note that the binomial degree of $B_{T_1, T_2}$ is equal to $1/2$ the sum of the lengths of all cycles with length $>2$ in $T_1 \cup T_2$.  Thus, we can proceed by considering two cases based on whether $T_1 \cup T_2$ has more than one cycle with length $>2$ or a single cycle of length $>2$.

For the first case, assume $T_1 \cup T_2$ has more than one cycle with length $>2$; let's call these cycles $\mathcal C_1, \ldots, \mathcal C_r$.  By Lemma \ref{lem:cycles},  $B_{T_1, T_2}$ can be written as the sum of $r$ binomials where the $i$th binomial can be factored into a monomial and the cycle binomial $B_{\mathcal C_i}$.  This means that the $i$th binomial in the sum has binomial degree equal to $1/2 \cdot $(length of $\mathcal C_i$), which is less than $k$.  Thus, by applying the induction hypothesis to each binomial in the sum, we have that $B_{T_1, T_2}$ is generated by quadratics and $B_{T_1, T_2} \in I_{flip}$.

For the second case, assume $T_1 \cup T_2$ has a single cycle $\mathcal C_1$ with length $>2$.  In this case, $\mathcal C_1$ is a perimeter cycle, Thus, by combining Lemma \ref{lem:perimetertocontractible} and Lemma \ref{lem:cycles}, and applying Lemma \ref{lem:contractiblecycles} to each resulting contractible cycle, we have $B_{T_1, T_2} \in I_{flip}$. 
\end{proof}

\begin{cor} Let $R$ be a $2$-dimensional simply connected cubiculated region.  Then $\mathcal T_R$ is flip connected.
\end{cor}

%Let $m$ and $n$ be fixed positive integers. We will denote the $2D$-box as $B_{m,n}$, which is a collection of $mn$ unit squares. We can think about a 2D-box as a homogeneous cubical complex of size $mn$. A cubical complex is $\ldots$, for more information on cubical complexes see .We will be working with subsets of $B_{m,n}$, obtained by keeping $k$ squares of $B_{m,n}$ as a cubical complex and discarding the remaining squares.  For such a subset $R_k$, we can define an associated grid graph $G_k$ in the same way that we defined $G_{m,n}$.  In particular, the vertices of $G_k$ are the squares of $R_k$ and edges are present whenever two square share an edge. Note that $G_k$ is a subgraph of the grid graph, $G_{m, n}$.

%\section{Further directions: higher dimensional cases}

%What is the 3D analog of a contractible cycle?

%{\color{magenta}  I think it is ok if we just end without the further directions}

%%%%%%FIRST BIG LEMMA FOR 2D%%%%%%%%%%%%%%%%%%%

%\begin{lemma}
%Assume that two edges in $T_1 \cup T_2$ are in the following configuration:
%\begin{center}
%    \includegraphics[scale=.4]{case4b.png}
%\end{center}
%Then, the edges $e$ and $f$ belong to two separate cycles in $\mathcal C_{T_1, T_2}$.
%\end{lemma}

%\begin{proof} {\color{blue} Do we need this lemma anymore?}
%\end{proof}
%\begin{figure}[h]
%    \includegraphics[scale=.4]{corner.png}
%    \caption{A corner of a $k$-homogeneous cubical complex}
%    \label{fig:corner}
%\end{figure}

\section*{Acknowledgments}
This paper would not have been possible without the support of the following individuals. We would like to extend gratitude to Dr. Caroline Klivans for inspiring us with her work on the connectivity of domino tiling spaces and lending her thoughts on our journey to proving our main theorem. Dr. Sylvie Corteel provided us constant support, and we thank her for sharing her unique insights for the $2$-dimensional case. We also would like to thank our friends at San Francisco State University, Dr. Serkan Ho\c{s}ten for thoughtful comments on an earlier draft of the manuscript and Dr. Federico Ardila for the many helpful conversations along the way. We thank Dr. Randy McCarthy for providing his insight as a topologist in addition to generating  examples to explore. 

\bibliographystyle{amsplain}
%\bibliographystyle{abbrvnat}
% use the following instead if you encounter problems 
%\bibliographystyle{alpha}
\bibliography{01Notes}
%%%%%%%%%%%%%%%%%%%%%%%%%%%%%%%%%%%%%%%%%%%%%
%\begin{thebibliography}{}%%%%%%%%%%%%%%%%%%%%%%%%%%%%%%%%%%%%%%%%%
%\raggedbottom%%%%%%%%%%%%%%%%%%%%%%%%%%%%%%%%%%%%%%%%%%%%%%%%%%%%%%%%%%%%

%%%%%%%%%%%%%%%%%%%%%%%%%%%%%%%%%%%%%%%%%%%%%%%%%%%%%%%%%%%%%%%%%%%%%%%%%
%\end{thebibliography}%%%%%%%%%%%%%%%%%%%%%%%%%%%%%%%%%%%%%%%%%%%%%%%%%%%%
%%%%%%%%%%%%%%%%%%%%%%%%%%%%%%%%%%%%%%%%%%%%%%%%%%%%%%%%%%%%%%%%%%%%%%%%%APPENDIX%%%%%%%%%%%%%%%%%%%%%%%

%\appendix
%\section{Chapter 4 Appendix}

%\end{itemize}

%%%%%%%%%%%%%%%%%%%%%%%%%%%%%%%%%%%%%%%%%%%%%%%%%%%%%%%%%%%%%%%%%%%%%%%%%
\end{document}